\documentclass{imsart}
\usepackage[utf8]{inputenc}
\usepackage[T1]{fontenc}
\usepackage{hyperref}
\usepackage{amsmath,amsthm} 
\usepackage{amssymb}
\usepackage{graphicx}
\usepackage{bbm}
\startlocaldefs
\usepackage{mathtools}
\newtheorem{thm}{Theorem}
\newtheorem{prop}{Proposition}
\newtheorem{lem}{Lemma}
\newtheorem{cor}{Corollary}
\newtheorem{Def}{Definition}
\usepackage{verbatim}
\def\Var{\mathop{\rm Var}}
\endlocaldefs

\begin{document}
\begin{frontmatter}
\title{Quantization/Clustering: when and why does \texorpdfstring{$k$}{k}-means work.}
\runtitle{When does $k$-means work?}

\begin{aug}
\author{\fnms{Cl\'{e}ment} \snm{ Levrard} \ead[label=e1]{levrard@math.univ-paris-diderot.fr}}
\runauthor{C. Levrard}
 \affiliation{LPMA, Universit\'e Paris Diderot, 8 place Aur\'e lie Nemours, 75013 Paris \printead{e1}}
\end{aug}

\begin{abstract}
Though mostly used as a clustering algorithm, $k$-means are originally designed as a quantization algorithm. Namely, it aims at providing a compression of a probability distribution with $k$ points. Building upon \cite{Levrard15,Tang16}, we try to investigate how and when these two approaches are compatible. Namely, we show that provided the sample distribution satisfies a margin like condition (in the sense of \cite{Tsybakov99} for supervised learning), both the associated empirical risk minimizer and the output of Lloyd's algorithm provide almost optimal classification in certain cases (in the sense of \cite{Azizyan13}). Besides, we also show that they achieved fast and optimal convergence rates in terms of sample size and compression risk.
\end{abstract}

\begin{keyword}[class=MSC]
\kwd{62H30}
\kwd{62E17}
\end{keyword}

\begin{keyword}
\kwd{$k$-means}
\kwd{clustering}
\kwd{quantization}
\kwd{separation rate}
\kwd{distortion}
\end{keyword}

\end{frontmatter}
\renewcommand{\c}{\mathbf{c}}
\newcommand{\m}{\mathbf{m}}
\newcommand{\ckm}{\hat{\c}_{KM,n}}
\section{Introduction}
           Due to its simplicity, $k$-means algorithm, introduced in \cite{McQueen67}, is one of the most popular clustering tool. It has been proved fruitful in many applications: as a last step of a spectral clustering algorithm \cite{Ng01}, for clustering electricity demand curves \cite{Antoniadis11}, clustering DNA microarray data \cite{Tavazoie99,Kim07} or EEG signals \cite{Orhan11} among others. As a clustering procedure, $k$-means intends to group data that are relatively similar into several well-separated classes. In other words, for a data set $\{X_1, \hdots, X_n\}$ drawn in a Hilbert space $\mathcal{H}$, $k$-means outputs $\hat{\mathcal{C}}=(C_1, \hdots, C_k)$ that is a collection of subsets of $\{1, \hdots n\}$.  To assess the quality of such a classification, it is often assumed that a target or natural classification $\mathcal{C}^*=(C^*_1, \hdots, C^*_k)$ is at hand. Then a classification error may be defined by
           \[
           \hat{R}_{classif}(\hat{\mathcal{C}},\mathcal{C}^*) = \inf_{\sigma \in \mathcal{S}_k} \frac{1}{n} \sum_{j=1}^k{ \left | \hat{C}_{\sigma(j)} \cap (C^*_j)^c \right |},
           \]
           where $\sigma$ ranges in the set of $k$-permutations $\mathcal{S}_k$. Such a target classification $\mathcal{C}^*$ may be provided by a mixture assumption on the data, that is hidden i.i.d latent variables $Z_1, \hdots, Z_n \in \{1, \hdots, k \}$ are drawn and only i.i.d $X_i$'s such that $X|Z=j \sim \phi_j$ are observed. This mixture assumption on the data is at the core of model-based clustering techniques, that cast the clustering problem into the density estimation framework. In this setting, efficient algorithms may be designed, provided that further assumptions on the $\phi_j$'s are made. For instance, if the $\phi_j$'s are supposed to be normal densities, this classification problem may be processed in practice using an EM algorithm \cite{Dempster77}. 
           
           However, by construction, $k$-means may rather be thought of as a quantization algorithm. Indeed, it is designed to output an empirical codebook $\hat{\c}_n = (\hat{c}_{n,1}, \hdots, \hat{c}_{n,k})$, that is a $k$-vector of codepoints $\hat{c}_{n,j} \in \mathcal{H}$, minimizing  
           \[
           \hat{R}_{dist}(\c) = \frac{1}{n} \sum_{i=1}^{n} \min_{j=1, \hdots, k} \| X_i - c_j \|^2,
           \]
           over the set of codebooks $\c = (c_1, \hdots, c_k)$. Let $V_i(\c)$ denote the $j$-th Voronoi cell associated with $\c$, that is $V_j(\c) = \left \{ x | \quad  \forall i \neq j \quad \|x-c_j\| \leq \|x-c_i\| \right \}$, and $Q_{\c}$ the function that maps every $V_j(\c)$ onto $\c_{j}$, with ties arbitrarily broken. Then $\hat{R}_{dist}(\c)$ is $P_n \|x - Q_{\c}(x)\|^2$, where $P_n f$ means integration with respect to the empirical distribution $P_n$. From this point of view, $k$-means aims at providing a quantizer $Q_{\hat{\c}_n}$ that realizes a good $k$-point compression of $P$, namely that has a low distortion $R_{dist}(\hat{\c}_n) = P \| x - Q_{\hat{\c}_n}(x)\|^2$. 
           
           This quantization field was originally developed to answer signal compression issues in the late $40$'s (see, e.g. \cite{Gersho91}), but quantization may also be used as a pre-processing step for more involved statistical procedures, such as modeling meta-models for curve prediction by $k$ ``local'' regressions as in \cite{Auder12}. This domain provides most of the theoretical results for $k$-means (see, e.g., \cite{Linder02, Biau08}), assessing roughly that it achieves an optimal $k$-point compression up to $1/\sqrt{n}$ in terms of the distortion $P \|x - Q_{\c}(x)\|^2$, under a bounded support assumption on $P$. Note that other distortion measures can be considered: $L_r$ distances, $r \geq 1$ (see, e.g., \cite{Graf00}), or replacing the squared Euclidean norm by a Bregman divergence (\cite{Fischer10}).
           
           In practice, $k$-means clustering is often performed using Lloyd's algorithm \cite{Lloyd82}. This iterative procedure is based on the following: from an initial codebook $\c^{(0)}$, partition the data according to the Voronoi cells of $\c^{(0)}$, then update the code point by computing the empirical mean over each cell. Since this step can only decrease the empirical distortion $\hat{R}_{dist}$, repeat until stabilization and output $\hat{\c}_{KM,n}$. Note that this algorithm is a very special case of the Classification EM algorithm in the case where the components are assumed to have equal and spherical variance matrices \cite{Celeux92}. As for EM's algorithms, the overall quality of the Lloyd's algorithm output mostly depends on the initialization. Most of the effective implementation use several random initializations, as for $k$-means $++$ \cite{Arthur07}, resulting in an approximation of the true empirical distortion minimizer. This approximation may be build as close as desired (in terms of distortion) to the optimum \cite{Kumar05} with high probability, provided that enough random initializations are allowed. 

Roughly, these approximation results quantify the probability that a random initialization falls close enough to the empirical distortion minimizer $\hat{\c}_n$. It has been recently proved that, provided such a good random initialization is found, if $P_n$ satisfies some additional clusterability assumption, then some further results on the misclassification error of the Lloyd's algorithm output can be stated. For instance, if $ \min_{i \neq j}{\| \hat{\c}_{n,i} - \hat{\c}_{n,j}} \|/\sqrt{n}$ is large enough, then it is proved that $\hat{\c}_{KM,n}$ provides a close classification to $\hat{\c}_n$ \cite{Tang16}. In other words, if $\mathcal{C}(\hat{\c}_{KM,n})$ and $\mathcal{C}(\hat{\c}_{n})$ denote the classifications associated with the Voronoi diagrams of $\hat{\c}_{KM,n}$ and $\hat{\c}_n$, then $\hat{R}_{classif}(\mathcal{C}(\hat{\c}_{KM,n}),\mathcal{C}(\hat{\c}_{n}))$ is small with high probability,  provided that the empirically optimal cluster centers are separated enough.

This \textit{empirical} separation condition has deterministic counterparts that provide classification guarantees for $k$-means related algorithms, under model-based assumptions. Namely, if the sample is drawn according to a subGaussian mixture, then a separation condition on the true means of the mixture entails guarantees for the classification error $\hat{R}_{classif}(\hat{\mathcal{C}},\mathcal{C}^*)$, where $\mathcal{C}^*$ is the latent variable classification \cite{Lu16, Bunea16}. As will be detailed in Section \ref{sec:notation}, it is possible to define a separation condition without assuming that the underlying distribution is a subGaussian mixture (see, e.g., \cite{Levrard13, Levrard15}). This so-called margin condition turns out to be satisfied under model-based clustering assumptions such as quasi-Gaussian mixtures. It also holds whenever the distribution is supported on finitely many points. 

Section \ref{sec:notation} introduces notation and basic structural properties that the margin condition entails for probability distributions. To be more precise, a special attention is paid to the connection between classification and compression such a condition provides. For instance, it is exposed that whenever $P$ satisfies a margin condition, there exist finitely many optimal classifications. Section  \ref{sec:convergence_erm}  focuses on the compression performance that an empirical risk minimizer $\hat{\c}_n$ achieves under this margin condition. We state that fast convergence rates for the distortion are attained, that imply some guarantees on the classification error of $\hat{\c}_n$. At last, Section \ref{sec:convergence_kmeans_algo} intends to provide similar results, both in compression and classification, for an output $\ckm$ of the Lloyd's algorithm. We show that our deterministic separation condition ensures that an empirical one in satisfied with high probability, allowing to connect our approach to that of \cite{Tang16}. On the whole, we prove that $\ckm$ performs almost optimal compression, as well as optimal classification in the framework of \cite{Azizyan13}.

\section{Notation and margin condition}\label{sec:notation}

Throughout this paper, for $M >0$ and $a$ in $\mathcal{H}$, $\mathcal{B}(a,M)$ will denote the closed ball with center $a$ and radius $M$. For a subset $A$ of $\mathcal{H}$, $\bigcup_{a \in A}{\mathcal{B}(a,M)}$ will be denoted by $\mathcal{B}(A,M)$. With a slight abuse of notation, $P$ is said to be $M$-bounded if its support is included in $\mathcal{B}(0,M)$. Furthermore, it will also be assumed that the support of $P$ contains more than $k$ points. Recall that we define the closed $j$-th Voronoi cell associated with $\c = (c_1, \hdots, c_k)$ by $V_j(\c) = \left \{ x | \quad  \forall i \neq j \quad \|x-c_j\| \leq \|x-c_i\| \right \}$.
         
         We let $X_1, \hdots, X_n$ be i.i.d. random variables drawn from a distribution $P$, and introduce the following contrast function, 
         \begin{align*}
          \gamma : \left \{
                 \begin{array}{@{}ccl@{}}
                 \left (\mathcal{H} \right )^k \times \mathcal{H}& \longrightarrow & \mathbb{R} \\
                 \hspace{0.45cm}  (\mathbf{c},x) & \longmapsto & \underset{j=1, \hdots, k}{\min}{\left \| x-c_j \right \|^2}
                 \end{array}
                 \right . ,
          \end{align*}
so that $R_{dist}(\c) = P \gamma (\c,.)$ and $\hat{R}_{dist}(\c) = P_n \gamma (\c,.)$. We let $\mathcal{M}$ denote the set of minimizers of $P \gamma(\c,.)$ (possibly empty). The most basic property of the set of minimizers is its stability with respect to isometric transformations that are $P$-compatible. Namely
\begin{lem}{\cite[Lemma 4.7]{Graf00}}
\label{lem:optimal_invariance}

Let $T$ be an isometric transformation such that $ T \sharp P = P$, where $T \sharp P$ denotes the distribution of $T(X)$, $X \sim P$. Then 
\[
T \left ( \mathcal{M} \right ) = \mathcal{M}.
\]
\end{lem}
 Other simple properties of $\mathcal{M}$ proceed from the fact that $\c \mapsto \|x-c_j\|^2$ is weakly lower semi-continuous (see, e.g., \cite[Proposition 3.13]{Brezis11}), as stated below.
\begin{prop}{\cite[Corollary 3.1]{Fischer10} and  \cite[Proposition 2.1]{Levrard15}}\label{prop:structure_minimizers} 

Assume that $P$ is $M$-bounded, then
\begin{itemize}
\item[$i)$] $\mathcal{M} \neq \emptyset$.
\item[$ii)$] If $B = \inf_{\c^* \in \mathcal{M}, i \neq j} \|c_i^* - c_j^*\|$, then $B >0$.
\item[$iii)$] If $p_{min} = \inf_{\c^* \in \mathcal{M},i} P(V_i(\c^*))$, then $p_{min} >0$. 
\end{itemize}
\end{prop}

Proposition \ref{prop:structure_minimizers} ensures that there exist minimizers of the true and empirical distortions $R_{dist}$ and $\hat{R}_{dist}$.
 In what follows, $\hat{\c}_n$ and $\c^*$ will denote minimizers of $\hat{R}_{dist}$ and $R_{dist}$ respectively. A basic property of distortion minimizers, called the centroid condition, is the following.
\begin{prop}{\cite[Theorem 4.1]{Graf00}}
\label{prop:centroid_condition}
  If $\c^*$ $\in$ $\mathcal{M}$, then, for all $j=1, \hdots, k$,
  \[
  {P(V_j(\c^*))} c_j^* = {P \left ( x \mathbbm{1}_{V_j(\c^*)}(x) \right )}.
  \]
  As a consequence, for every $\c \in \mathcal{H}^k$  and $\c^* \in \mathcal{M}$,
  \[
  R_{dist}(\c) - R_{dist}(\c^*) \leq \|\c - \c^*\|^2.
  \]
\end{prop}
  A direct consequence of Proposition \ref{prop:centroid_condition} is that the boundaries of the Voronoi diagram $V(\c)$ has null $P$-measure. Namely, if 
  \[
  N(\c^*) = \bigcup_{ i \neq j} \left \{ x | \quad \|x-c_i^*\|= \|x - c_j^* \| \right \},
  \]
  then $P(N(\c^*)) = 0$. Hence the quantizer $Q_{\c^*}$ that maps $V_j(\c^*)$ onto $c_j^*$ is well-defined $P$ a.s. For a generic $\c$ in $\mathcal{B}(0,M)$, this is not the case. Thus, we adopt the following convention: $W_1(\c) = V_1(\c)$, $W_2(\c) = V_2(\c) \setminus W_1(\c)$, $\hdots$, $W_k(\c) = V_k(\c) \setminus W_{k-1}(\c)$, so that the $W_j(\c)$'s form a tessellation of $\mathbb{R}^d$. The quantizer $Q_{\c}$ can now be properly defined as the map that sends each $W_j(\c)$ onto $c_j$. As a remark, if $Q$ is a $k$-points quantizer, that is a map from $\mathbb{R}^d$ with images $c_1, \hdots, c_k$, then it is immediate that $R_{dist}(Q) \geq R_{dist}(Q_\c)$. This shows that optimal quantizers in terms of distortion are to be found among nearest-neighbor quantizers of the form $Q_\c$, $\c$ in $(\mathbb{R}^d)^k$. 
  
     An other key parameter for quantization purpose is the separation factor, that seizes the difference between local and global minimizers in terms of distortion.
     
     \begin{Def}\label{def:separation_factor} Denote by $\bar{\mathcal{M}}$ the set of codebooks that satisfy
     \begin{align*}
     P  \left ( W_i(\c) \right )c_i = P \left ( x \mathbbm{1}_{W_i(\c)}(x) \right ),
     \end{align*}
     for any $i = 1, \hdots, k$. Let $\varepsilon >0$, then $P$ is said to be $\varepsilon$-separated if 
     \[
     \inf_{ \c \in \bar{\mathcal{M}} \setminus \mathcal{M}} R_{dist}(\c) - R_{dist}(\c^*) \geq \varepsilon,
     \]
     where $\c^* \in \mathcal{M}$.     
     \end{Def}
The separation factor $\varepsilon$ quantifies how difficult the identification of global minimizer might be. Its empirical counterpart in terms of $\hat{R}_{dist}$ can be thought of as the minimal price one has to pay when the Lloyd's algorithm ends up at a stationary point that is not an optimal codebook.

    Note that local minimizers of the distortion satisfy the centroid condition, as well as $p$-optimal codebooks, for $p<k$. Whenever $\mathcal{H}=\mathbb{R}^d$, $P$ has a density and $P\|x\|^2 < \infty$, it can be proved that the set of minimizers of $R_{dist}$ coincides with the set of codebooks satisfying the centroid condition, also called stationary points (see, e.g., Lemma A of \cite{Pollard82}). However, this result cannot be extended to non-continuous distributions, as proved in Example 4.11 of \cite{Graf00}.
   
  Up to now, we only know that the set of minimizers of the distortion $\mathcal{M}$ is non-empty. From the compression point of view, this is no big deal if $\mathcal{M}$ is allowed to contain an infinite number of optimal codebooks. From the classification viewpoint, such a case may be interpreted as a case where $P$ carries no natural classification of $\mathcal{H}$. For instance, if $\mathcal{H}=\mathbb{R}^2$ and $P \sim \mathcal{N}(0,I_2)$, then easy calculation and Lemma \ref{lem:optimal_invariance} show that $\mathcal{M} = \left \{ (c_1,c_2) | \quad c_2 = -c_1, \quad \|c_1\|= 2/\sqrt{2\pi} \right \}$, hence $\left | \mathcal{M} \right | = + \infty$. In this case, it seems quite hard to define a natural classification of the underlying space, even if the $\c^*$'s are clearly identified. The following margin condition is intended to depict situations where a natural classification related with $P$ exists.
   \begin{Def}[Margin condition]\label{def:margincondition}
         A distribution $P$ satisfies a margin condition with radius $r_0 >0$ if and only if
				\begin{itemize}
				\item[$i)$] $P$ is $M$-bounded,
				\item[$ii)$] for all $0 \leq t \leq r_0$,
					\begin{align}\label{majorationkappa}
					 \sup_{\c^* \in \mathcal{M}} P \left ( \mathcal{B}(N(\c^*),t) \right ):=p(t)  \leq \frac{B p_{min}}{128 M^2}t.
					\end{align}
				\end{itemize}
         \end{Def}
         Since $p(2M)=1$, such a $r_0$ must satisfy $r_0 < 2M$. The constant $1/128$ in \eqref{majorationkappa} is not optimal and should be understood as a small enough absolute constant. The margin condition introduced above asks that every classification associated with an optimal codebook $\c^*$ is a somehow natural classification. In other words $P$ has to be concentrated enough around each $c_j^*$. This margin condition may also be thought of as a counterpart of the usual margin conditions for supervised learning stated in \cite{Tsybakov99}, where the weight of the neighborhood of the critical area $\left \{ x | \quad P (Y=1|X=x) =1/2 \right \}$ is controlled.

         The scope of the margin condition allows to deal with several very different situations in the same way, as illustrated below.
         \subsection{Some instances of 'natural classifications'}\label{subsec:MC_instances}
         
\noindent\textbf{Finitely supported distributions}:         
         If $P$ is supported on finitely many points, say $x_1, \hdots, x_r$. Then, $\mathcal{M}$ is obviously finite. Since, for all $\c^*$ in $\mathcal{M}$, $P(N(\c^*))=0$, we may deduce that $\inf_{\c^*, j} d(x_j,N(\c^*)) = r_0 > 0$. Thus, $p(t)=0$ for $t\leq r_0$, and $P$ satisfies a margin condition with radius $r_0$.   
     \medskip    
                  
\noindent\textbf{Truncated Gaussian mixtures}:        A standard assumption assessing the existence of a natural classification is the Gaussian mixture assumption on the underlying distribution, that allows to cast the classification issue into the density estimation framework. Namely, for $\mathcal{H}= \mathbb{R}^d$, $\tilde{P}$ is a Gaussian mixture if it has density
        \[
                 \tilde{f}(x) = \sum_{i=1}^{k}{\frac{\theta_i}{ (2 \pi)^{d/2} \sqrt{ \left | \Sigma_i \right | }}e^{-\frac{1}{2}(x-m_i)^t \Sigma_i^{-1} (x-m_i)}},
                 \]
                 where the $\theta_i$'s denote the weights of the mixture, the $m_i$'s the means and the $\Sigma_i$'s are the $d \times d$ covariance matrices of the components. 
                 
                 Also denote by $\tilde{B} = \min_{i \neq j}{\|m_i - m_j\|}$ the minimum distance between two components, and by $\sigma^2$ and $\sigma_-^2$ the largest and smallest eigenvalues of the $\Sigma_i$'s. It seems natural that the larger $\tilde{B}$ is compared to $\sigma$, the easier the classification problem would be. To this aim, we may define, for $\mathcal{C}$ and $\mathcal{C}^*$ two classifications the classification risk as the probability that a random point is misclassified, that is 
       \begin{align*}
       R_{classif}(\mathcal{C},\mathcal{C}^*) = \inf_{\sigma \in \mathcal{S}_k}{P \left ( \bigcup_{j =1}^{k} C_{\sigma(j)} \cap (C^*_j)^c \right )}.
       \end{align*}
        In the case $k=2$, $\theta_i = 1/2$ and $\Sigma_i = \sigma^2 I_d$,  \cite[Theorem 1 and 2]{Azizyan13} show that
                 \[
                 \inf_{\hat{C}} \sup_{ \sigma/\tilde{B} \leq \kappa} \mathbb{E} R_{classif}(\hat{C},\mathcal{C}^*) \asymp \kappa^2 \sqrt{\frac{d}{n}},
                 \]
                 up to $\log$ factors, where $\mathcal{C}^*$ denote the Bayes classification. Note that in this case, the Bayes classification is given by $C^*_j = V_j(\m)$, that is the Voronoi diagram associated with the vector of means.                
       Similarly we will show that for $\sigma/\tilde{B}$ small enough, a margin condition is satisfied.
                 
                 Since Gaussian mixture have unbounded distributions, we may define a truncated Gaussian mixture distribution by its density of the form
                 \[
                 \tilde{f}(x) = \sum_{i=1}^{k}{\frac{\theta_i}{ (2 \pi)^{d/2} N_i \sqrt{ \left | \Sigma_i \right | }}e^{-\frac{1}{2}(x-m_i)^t \Sigma_i^{-1} (x-m_i)}} \mathbbm{1}_{\mathcal{B}(0,M)}(x),
                \]
                where $N_i$ denotes a normalization constant for each truncated Gaussian variable. To avoid boundary issues, we will assume that $M$ is large enough so that $M \geq 2 \sup_{j} \|m_j\|$. On the other hand, we also assume that $M$ scales with $\sigma$, that is $M \leq c \sigma $, for some constant $c$. In such a setting, the following hold.
                \begin{prop}\label{prop:GM_margincondition}
                 Denote by $\eta = \min_{i}{1 - N_i}$. Then there exists constants $c_1(k,\eta,d,\theta_{min})$ and $c_2(k,\eta,d,\theta_{min},c_-,c)$ such that
                \begin{itemize}
                \item If $\sigma/\tilde{B} \leq \frac{1}{16 c_1 \sqrt{d}}$, then for all $j$ and $\c^*$ in $\mathcal{M}$, $\|c^*_j - m_j\| \leq c_1 \sigma \sqrt{d}$. 
                \item Assume that $\sigma_- \geq c_- \sigma$, for some constant $c_-$. If $\sigma/\tilde{B} \leq c_2$, then $\c^*$ is unique and $P$ satisfies a margin condition with radius $\tilde{B}/8$.
                \end{itemize}
                A possible choice of $c_1$ is $\sqrt{\frac{k2^{d+2}}{(1-\eta)\theta_{min}}}$.
                \end{prop}
                
                A short proof is given in Section \ref{subsec:proof_prop_GM_margincondition}. Proposition \ref{prop:GM_margincondition} entails that (truncated) Gaussian mixtures are in the scope of the margin condition, provided that the components are well-separated. As will be detailed in Section \ref{sec:convergence_kmeans_algo}, this implies that under the conditions of Proposition \ref{prop:GM_margincondition} the classification error of the outputs of the $k$-means algorithm is of order $\kappa ^2 \sqrt{d/n}$ as in \cite{Azizyan13}.  
                
                \subsection{An almost necessary condition}
                
                As described above, if the distribution $P$ is known to carry a natural classification, then it is likely that it satisfies a margin condition. It is proved below that conversely an optimal codebook $\c^*$ provides a not so bad classification, in the sense that the mass around $N(\c^*)$ must be small. To this aim, we introduce, for $\c$ in $\mathcal{B}(0,M)^k$, and $i\neq j$, the following mass
                \[
                p_{ij}(\c,t) = P \left ( \left \{ x | \quad 0 \leq \left\langle x - \frac{c_i + c_j}{2}, \frac{c_j-c_i}{r_{i,j}(\c)} \right\rangle \leq t \right \} \cap V_j(\c) \right ),
                \]
                where $r_{i,j}(\c) = \|c_i - c_j\|$. It is straightforward that $P(\mathcal{B}(N(\c),t)) \leq \sum_{i \neq j} p_{i,j}(\c,t)$. The necessary condition for optimality in terms of distortion is the following.
                \begin{prop}\label{prop:CN_optimality}
                Suppose that $\c^* \in \mathcal{M}$. Then, for all $i \neq j$ and $t < 1/2$, 
                \begin{align*}
                \int_{0}^{t r_{i,j}(\c^*)}{p_{i,j}(\c^*,s)ds} &\leq  2 t^2 r_{i,j}(\c^*) \left [ \frac{p_i(\c^*)}{1-2t} \wedge \frac{p_j(\c^*)}{1+2t} \right ], \\
                \int_{0}^{t r_{i,j}(\c^*)}p_{i,j}(\c^*,s)ds & \leq t^2 r_{i,j}(\c^*) \frac{p_i(\c^*) + p_j(\c^*)}{2},
                \end{align*}
                where $p_j(\c^*)$ denotes $P(V_j(\c^*))$.
                \end{prop}
                A proof of Proposition \ref{prop:CN_optimality} is given in Section \ref{subsec:proof_CN_optimality}. Whenever $p_{i,j}(\c^*,.)$ is continuous, Proposition \ref{prop:CN_optimality} can provide a local upper bound on the mass around $N(\c^*)$. 
                \begin{cor}\label{cor:MC_necessary}
                Assume that $\c^* \in \mathcal{M}$ and, for all $i \neq j$ and $t \leq t_0$ $p_{i,j}$ is continuous on $[0,t_0]$. Then there exists $r_0 >0$ such that, for all $r \leq r_0$,
                \begin{align*}
                P(\mathcal{B}(N(\c^*),r)) \leq \frac{8k}{B}r.                
                \end{align*}                
                \end{cor} 
Note that whenever $\mathcal{H}=\mathbb{R}^d$ and $P$ has a density, the assumptions of Corollary \ref{cor:MC_necessary} are satisfied. In this case,
Corollary \ref{cor:MC_necessary} states that all optimal codebooks satisfy a condition that looks like Definition \ref{def:margincondition}, though with a clearly worse constant than the required one.  Up to a thorough work on the constants involved in those results, this suggests that margin conditions (or at least weaker but sufficient versions) might be quite generally satisfied. As exposed below, satisfying such a condition provides interesting structural results.   
         \subsection{Structural properties under margin condition}
         
         The  existence of a natural classification, stated in terms of a margin condition in Definition \ref{def:margincondition}, gives some guarantees on the set of optimal codebooks $\mathcal{M}$. Moreover, it also allows local convexity of the distortion $R_{dist}$. These properties are summarized in the following fundamental Proposition.
         \begin{prop}{\cite[Proposition 2.2]{Levrard15}}\label{prop:properties_under_MC}
         Assume that $P$ satisfies a margin condition with radius $r_0$, then the following properties hold.
         \begin{itemize}
         \item[$i)$] For every $\c^*$ in $\mathcal{M}$ and $\c$ in $\mathcal{B}(0,M)^k$, if $\| \c - \c^* \| \leq \frac{Br_0}{4 \sqrt{2}M}$, then
          \begin{align}\label{eq:localconvexity}
			 R_{dist}(\c) - R_{dist}(\c^*) \geq \frac{p_{min}}{2} \| \c - \c^*\|^2.
				\end{align}
         \item[$ii)$] $\mathcal{M}$ is finite.
         \item[$iii)$] There exists $\varepsilon >0$ such that $P$ is $\varepsilon$-separated.
         \item[$iv)$] For all $\c$ in $\mathcal{B}(0,M)^k$,
         \begin{align}\label{eq:lienantos}
         \frac{1}{16M^2} \Var(\gamma(\c,.) - \gamma(\c^*(\c),.)) \leq \| \c - \c^*(\c) \|^2 \leq \kappa_0 \left ( R_{dist}(\c) - R_{dist}(\c^*) \right ),
         \end{align}
         where $\kappa_0 = 4kM^2\left ( \frac{1}{\varepsilon} \vee \frac{64 M^2}{p_{min} B^2 r_0^2} \right )$, and $\c^*(\c) \in \underset{ \c^* \in \mathcal{M}}{\arg\min}{\| \c - \c^*\|}$.
         \end{itemize}
         \end{prop}
         
         Properties $ii)$ and $iii)$ guarantee that whenever a margin condition is satisfied, there exist finitely many optimal codebooks that are clearly separated in terms of distortion. When $P \sim \mathcal{N}(0,I_d)$, since $|\mathcal{M}| = + \infty$, $P$ does not satisfy a margin condition. This finite set property also allows to give some structural results about the optimal codebooks. Namely, we can easily deduce the following.
         \begin{cor}\label{cor:optimal_codebook_invariance}
         Let $\mathcal{T}$ be the isometry group of $P$, and let $\c^* \in \mathcal{M}$. If $P$ satisfies a margin condition, then $| \mathcal{T}(\c^*) | < + \infty$.
         \end{cor}
An easy instance of application of Corollary \ref{cor:optimal_codebook_invariance} can be stated  in the truncated Gaussian Mixture model exposed  in Section \ref{subsec:MC_instances}. Let $S(\m)$ denote the subset of $\{1, \hdots, d \}$ such that, for all $j$ and $r \notin S(\m)$ $m_j^{(r)}=0$, where $m_j^{(r)}$ denotes the $r$-th coordinate of $m_j$. Under the conditions of Proposition \ref{prop:GM_margincondition}, if we further require that for all $j$ and $r,s$ in $S(\m) \times S(\m)^c$, $\Sigma_{j,rs}=0$, then it is immediate that $S(\c^*) \subset S(\m)$. Such a property might be of particular interest when variable selection is performed as in \cite{Levrard18}.   

       Properties $i)$ and $iv)$ of Proposition \ref{prop:properties_under_MC} allow to make connections between the margin condition defined in Definition \ref{def:margincondition} and earlier results on improved convergence rates for the distortion. To be more precise, it is proved in \cite{Chou94} that if $P$ has a continuous density, unique optimal codebook $\c^*$, and if the distortion function $R_{dist}$ has a positive Hessian matrix at $\c^*$, then $ R_{dist}(\hat{\c}_n) - R_{dist}(\c^*)  = O_{\mathbb{P}}(1/n)$. It is straightforward that in the case where $P$ has a continuous density and a unique optimal codebook, \eqref{eq:localconvexity} yields that the Hessian matrix of the distortion is positive, hence the margin condition gives the convergence rate in $O_{\mathbb{P}}(1/n)$ for the distortion in this case.
       
       On the other hand, it is proved in \cite[Theorem 2]{Antos04} that, if $ \Var(\gamma(\c,.) - \gamma(\c^*(\c),.) \leq A (R_{dist}(\c) - R_{dist}(\c^*))$, for some constant $A$, then the convergence rate $\mathbb{E}(R_{dist}(\hat{\c}_n) - R_{dist}(\c^*)) \leq C/n$ can be attained for the expected distortion of an empirical distortion minimizer. Thus, if $P$ satisfies a margin condition, then \eqref{eq:lienantos} shows that $P$ is in the scope of this result. In the following section, more precise bounds are derived for this excess distortion when $P$ satisfies a margin condition.
       
       At last, Properties $i)$ and $iv)$ allow to relate excess distortion and excess classification risk, when appropriate. 
       For a codebook $\c$ in $\mathcal{H}^k$, we denote by $\mathcal{C}(\c)$ its associated Voronoi partition (with ties arbitrarily broken). 
       \begin{cor}\label{cor:link_distortion_classification}
       Assume that $P$ satisfies a margin condition (Definition \ref{def:margincondition}) with radius $r_0$. Let 
 $\delta$ denote the quantity $\frac{p_{min}B^2 r_0^2}{64 M^2} \wedge \varepsilon$. For every $\c \in \mathcal{H}^k$ such that $R_{dist}(\c) - R_{dist}(\c^*) \leq \delta$, we have 
       \begin{align*}
       R_{classif}(\mathcal{C}(\c),\mathcal{C}(\c^*(\c))) \leq \frac{\sqrt{p_{min}}}{16M}  \sqrt{R_{dist}(\c) - R_{dist}(\c^*)},  
\end{align*}
where $\c^*(\c)$ is a closest optimal codebook to $\c$.        
      \end{cor}
A short proof of Corollary \ref{cor:link_distortion_classification} is given in Section \ref{subsec:proof_cor_link_distortion_classification}. Corollary \ref{cor:link_distortion_classification} summarizes the connection between classification and distortion carried by the margin condition: if a natural classification exists, that is if $P$ is separated into $k$ spherical components, then this classification can be inferred from quantizers that are designed to achieve a low distortion. As exposed in the following section, an other interest in satisfying a margin condition is achieving an improved convergence rate in terms of distortion for the empirical distortion minimizer.        
      
\section{Convergence of an empirical risk minimizer}\label{sec:convergence_erm}

           If $P$ is $M$-bounded, then the excess distortion of an empirical distortion minimizer can be bounded by
           \[
           \mathbb{E} ( R_{dist}(\hat{\c}_n) - R_{dist}(\c^*)) \leq \frac{C(k) M^2}{\sqrt{n}}.
           \]
      Such a result can be found in \cite{Linder02} for the case $\mathcal{H}=\mathbb{R}^d$, and in \cite{Biau08} for the general case where $\mathcal{H}$ is a separable Hilbert space. When $P$ satisfies a margin condition, faster rates can be achieved. The following Theorem is a refined version of \cite[Theorem 3.1]{Levrard15}.
      \begin{thm}\label{thm:fast_rates_distortion}
      We assume that $P$ satisfies a margin condition (Definition \ref{def:margincondition}) with radius $r_0$, and we let $\delta$ denote the quantity $\frac{p_{min}B^2 r_0^2}{64 M^2} \wedge \varepsilon$. Then, 
      \begin{align*}
        \mathbb{E} \left ( R_{dist}(\hat{\c}_n) - R_{dist}(\c^*) \right ) \leq \begin{multlined}[t] \frac{ C\left ( k + \log \left ( \left | \bar{\mathcal{M}} \right | \right ) \right )M^2 }{np_{min}} + \left [ \frac{20 k M^2}{\sqrt{n}} - \delta \right ] \mathbbm{1}_{\delta < \frac{12kM^2}{\sqrt{n}}}  \\ + \left [ e^{- \frac{n}{2M^4}\left ( (\delta - \frac{12kM^2}{\sqrt{n}}) \right )^2} \frac{M^2}{\sqrt{n}} \right ] \mathbbm{1}_{\delta \geq \frac{12kM^2}{\sqrt{n}}},
         \end{multlined}
         \end{align*}
         where $C$ denotes a (known) constant and $\left | \bar{\mathcal{M}} \right |$ denotes the number of optimal codebooks up to relabeling.
      \end{thm}      
      
      A short proof is given in Section \ref{subsec:proof_thm_fast_rates_distortion}. Theorem \ref{thm:fast_rates_distortion} confirms that the fast $1/n$ rate for the distortion may be achieved as in \cite{Pollard82} or \cite{Antos04}, under slightly more general conditions. It also emphasizes that the convergence rate of the distortion is ``dimension-free'', in the sense that it only depends on the dimension through the radius of the support $M$. For instance, quantization of probability distributions over the unit $L_2$-ball of $L_2([0,1])$ (squared integrable functions) is in the scope of Theorem \ref{thm:fast_rates_distortion}. Note that a deviation bound is also available for $R_{dist}(\hat{\c}_n) - R_{dist}(\c^*)$, stated as \eqref{eq:deviation_proba_in_basin}.
      
       In fact, this result shows that the key parameters that drive the convergence rate are rather the minimal distance between optimal codepoints $B$, the margin condition radius $r_0$ and the separation factor $\varepsilon$. These three parameters provide a local scale $\delta$ such that, if $n$ is large enough to distinguish codebooks at scale $\delta$ in terms of slow-rated distortion, i.e. $\sqrt{n} \delta \geq 12 k M^2$, then the distortion minimization boils down to $k$ well separated mean estimation problems, leading to an improved convergence rate in $k M^2/(n p_{min})$. Indeed, Theorem \ref{thm:fast_rates_distortion} straightforwardly entails that, for $n$ large enough, 
      \[
      \mathbb{E} \left ( R_{dist}(\hat{\c}_n) - R_{dist}(\c^*) \right ) \leq \frac{C'(k + \log(|\bar{\mathcal{M}}|) M^2}{n p_{min}}.
      \]
     Thus, up to the $\log(| \bar{\mathcal{M}} |)$ factor, the right-hand side corresponds to $\sum_{j=1}^{k} \mathbb{E}\left ( \| X - c_j^* \|^2 | X \in V_j(\c^*) \right )$. Combining Theorem \ref{thm:fast_rates_distortion} and Corollary \ref{cor:link_distortion_classification} leads to the following classification error bound for the empirical risk minimizer $\hat{\c}_n$. Namely, for $n$ large enough, it holds
     \begin{align*}
     \mathbb{E} \left [R_{classif}(\mathcal{C}(\hat{\c}_n),\mathcal{C}(\c^*(\hat{\c}_n))) \right ] \leq C' \frac{\sqrt{k +\log(|\bar{\mathcal{M}}|}}{\sqrt{n}}.
     \end{align*}     
     This might be compared with the $1/\sqrt{n}$ rate obtained in \cite[Theorem 1]{Azizyan13} for the classification error under Gaussian mixture with well-separated means assumption. Note however that in such a framework $\mathcal{C}(\c^*)$ might not be the optimal classification. However, under the assumptions of Proposition \ref{prop:GM_margincondition}, $\mathcal{C}(\c^*)$ and $\mathcal{C}(\m)$ can be proved close, and even the same in some particular cases as exposed in Corollary \ref{cor:classification_error_GM}.
     
     Next we intend to assess the optimality of the convergence rate exposed in Theorem \ref{thm:fast_rates_distortion}, by investigating lower bounds for the excess distortion over class of distributions that satisfy a margin condition. We let $\mathcal{D}(B_-,r_{0,-},p_-,\varepsilon_-)$ denote the set of distributions satisfying a margin condition with parameters      $B \geq B_-$, $r_0 \geq r_{0,-}$, $p_{min} \geq p_-$ and $\varepsilon \geq \varepsilon_-$. Some lower bound on the excess distortion over these sets are stated below.
     
     \begin{prop}{\cite[Proposition 3.1]{Levrard15}}\label{prop:minimax_old}
     If $\mathcal{H}= \mathbb{R}^d$, $k \geq 3$ and $n \geq 3k/2$, then 
     \begin{align*}
     \inf_{\hat{\c}} \sup_{P \in \mathcal{D}(c_1Mk^{-1/d},c_2Mk^{-1/d},c_3/k,c_4M^2 k^{-2/d}/\sqrt{n})}\mathbb{E} \left [ R_{dist}(\hat{\c}) - R_{dist}(\c^*) \right ] \geq c_0 \frac{M^2 k^{\frac{1}{2} - \frac{1}{d}}}{\sqrt{n}},
     \end{align*}
     where $c_0$, $c_1, c_2, c_3$ and $c_4$ are absolute constants.
     \end{prop}
     Thus, for a fixed choice of $r_0$, $B$ and $p_{min}$, the upper bound given by Theorem \ref{thm:fast_rates_distortion} turns out to be optimal if the separation factor $\varepsilon$ is allowed to be arbitrarily small (at least $\delta \lesssim kM^2/\sqrt{n}$).   When all these parameters are fixed, the following Proposition \ref{prop:minimax_new} ensures that the $1/n$ rate is optimal.
     
     \begin{prop}\label{prop:minimax_new}
     Let $d=dim(\mathcal{H})$. Assume that $ n \geq k$, then there exist constants $c_1$, $c_2$, $c_3$ and $c_0$ such that 
     \begin{align*}
     \inf_{\hat{\c}} \sup_{P \in \mathcal{D}(c_1Mk^{-1/d},c_2Mk^{-{1/d}},1/k,c_3M^2 k^{-(1+2/d})} \mathbb{E} \left [ R_{dist}(\hat{\c}) - R_{dist}(\c^*) \right ] \geq c_0 \frac{M^2k^{1- \frac{2}{d}}}{n}.
     \end{align*}
     \end{prop}
A proof of Proposition \ref{prop:minimax_new} can be found in Section \ref{subsec:minimax_proofs}. Proposition \ref{prop:minimax_new} ensures that the $1/n$-rate is optimal on the class of distributions satisfying a margin condition with fixed parameters. Concerning the dependency in $k$, note that Proposition \ref{prop:minimax_new} allows for $d = + \infty$, leading to a lower bound in $k$. In this case the lower bound differs from the upper bound given in Theorem \ref{thm:fast_rates_distortion} up to a $1/p_{min} \sim k$ factor. A question raised by the comparison of Proposition \ref{prop:minimax_old} and Proposition \ref{prop:minimax_new} is the following: can we retrieve the $1/\sqrt{n}$ rate when allowing other parameters such as $B_-$ or $r_{0,-}$ to be small enough and $\varepsilon_-$ fixed? A partial answer is provided by the following structural result, that connects the different quantities involved in the margin condition.

\begin{prop}\label{prop:connections_B_epsilon_r0}
Assume that $P$ satisfies a margin condition with radius $r_0$. Then the following properties hold.
\begin{itemize}
\item[$i)$]   $\varepsilon \leq \frac{B^2}{4}$.
\item[$ii)$] $r_0 \leq B$.
\end{itemize}
\end{prop}

A proof of Proposition \ref{prop:connections_B_epsilon_r0} is given in Section \ref{subsec:proof_prop_connections_B_epsilon_r0}. Such a result suggests that finding distributions that have $B$ small enough whereas $\varepsilon$ or $r_0$ remains fixed is difficult. As well, it also indicates that the separation rate in terms of $B$ should be of order $M k^{-1/d} n^{-1/4}$. Slightly anticipating, this can be compared with the $n^{-1/4}$ rate for the minimal separation distance between two means of a Gaussian mixture to ensure a consistent classification, as exposed in \cite[Theorem 2]{Azizyan13}. 

\section{Convergence of the $k$-means algorithm}\label{sec:convergence_kmeans_algo}

      Up to now some results have been stated on the performance of an empirical risk minimizer $\hat{\c}_n$, in terms of distortion or classification. Finding such a minimizer is in practice intractable (even in the plane this problem has been proved $NP$-hard, \cite{Mahajan12}). Thus, most of $k$-means algorithms provide an approximation of such a minimizer. For instance, Lloyd's algorithm outputs a codebook $\hat{\c}_{KM,n}$ that is provably only a stationary point of the empirical distortion $\hat{R}_{dist}$. Similarly to the EM algorithm, such a procedure is based on a succession of iterations that can only decrease the considered empirical risk  $\hat{R}_{dist}$. Thus many random initializations are required to ensure that at least one of them falls into the basin of attraction of an empirical risk minimizer. 
      
      Interestingly, when such a good initialization has been found, some recent results ensure that the output $\ckm$ of Lloyd's algorithm achieves good classification performance, provided that the sample is in some sense well-clusterable. For instance, under the model-based assumption that $X$ is a mixture of sub-Gaussian variables with means $\m$ and maximal variances $\sigma^2$, \cite[Theorem 3.2]{Lu16} states that, provided $\tilde{B}/\sigma$ is large enough, after more that $4\log(n)$ iterations from a good initialization Lloyd's  algorithm outputs a codebook with classification error less that $e^{-\tilde{B}^2/(16 \sigma^2)}$. Note that the same kind of results hold for EM-algorithm in the Gaussian mixture model, under the assumption that $\tilde{B}/\sigma$ is large enough and starting from a good initialization (see, e.g., \cite{Dasgupta07}).
      
      In the case where $P$ is not assumed to have a mixture distribution, several results on the classification risk $\hat{R}_{classif}(\ckm,\hat{\c}_n)$ are available, under clusterability assumptions.
  Note that this risk accounts for the misclassifications encountered by the output of Lloyd's algorithm compared to the empirical risk minimizer, in opposition to a latent variable classification as above.
  \begin{Def}{\cite[Definition 1]{Tang16}}\label{def:clusterability}
  A sample $X_1, \hdots, X_n$ is $f$-clusterable if there exists a minimizer $\hat{\c}_n$ of $\hat{R}_{dist}$ such that, for $j\neq i$,
  \begin{align*}
  \|\hat{\c}_{n,i} - \hat{\c}_{n,j}\| \geq f \sqrt{\hat{R}_{dist}(\hat{\c}_n}) \left ( \frac{1}{\sqrt{n_i}} + \frac{1}{\sqrt{n_j}} \right ),
  \end{align*}
  where $n_\ell$ denotes $| \left \{ i| \quad  X_i \in V_\ell(\hat{\c}_n) \right \}|$.
  \end{Def}     
     It is important to mention that other definitions of clusterability might be found, for instance in \cite{Kumar10,Awashti12}, each of them requiring that the optimal empirical codepoints are well-separated enough. Under such a clusterability assumption, the classification error of $\ckm$ can be proved small provided that a good initialization is chosen.
     \begin{thm}{\cite[Theorem 2]{Tang16}}\label{thm:classif_error_km}
     Assume that $X_1, \hdots, X_n$ is $f$-clusterable, with $f >32$ and let $\hat{\c}_n$ denote the corresponding minimizer of $\hat{R}_{dist}$. Suppose that the initialization codebook $\c^{(0)}$ satisfies 
     \[
     \hat{R}_{dist}(\c^{(0)}) \leq g \hat{R}_{dist}(\hat{\c}_n),
     \]
     with $g < \frac{f^2}{128} -1$. Then the outputs of Lloyd's algorithm satisfies
     \[
     \hat{R}_{classif}(\ckm,\hat{\c}_n) \leq \frac{81}{8f^2}.
     \]
\end{thm}      
The requirement on the initialization codebook $\c^{(0)}$ is stated in terms of $g$-approximation of an empirical risk minimizer. Finding such approximations can be carried out using approximated $k$-means techniques ($k$-means $++$ \cite{Arthur07}), usual clustering algorithms (single Linkage \cite{Tang16}, spectral clustering \cite{Lu16}), or even more involved procedures as in \cite{Ostrovsky12} coming with complexity guarantees. All of them entail that a $g$-approximation of an empirical risk minimizer can be found with high probability (depending on $g$), that can be used as an initialization for the Lloyd's algorithm. 

Interestingly, the following Proposition allows to think of Definition \ref{def:clusterability} as a margin condition (Definition \ref{def:margincondition}) for the empirical distribution.

\begin{prop}\label{prop:connection_clusterability_empiricalmargin}
Let $\hat{p}(t)$, $\hat{B}$ and $\hat{p}_{min}$ denote the empirical counterparts of $p(t)$, $B$ and $p_{min}$. If
\[
\hat{p}\left( \frac{16M^2 f}{\sqrt{n \hat{p}_{min}}\hat{B}} \right ) \leq \hat{p}_{min},
\]
then $X_1, \hdots, X_n$ is $f$-clusterable.
\end{prop}
A proof of Proposition \ref{prop:connection_clusterability_empiricalmargin} can be found in Section \ref{sec:Proof of Proposition prop:connection_clusterability_empiricalmargin}. Intuitively, it seems likely that if $X_1, \hdots, X_n$ is drawn from a distribution $P$ that satisfies a margin condition, then $X_1, \hdots, X_n$ is clusterable in the sense of Definition \ref{def:clusterability}. This is formalized by the following  Theorem.

\begin{thm}\label{thm:margin_condition_gives_clusterability}
Assume that $P$ satisfies a margin condition. Let $p>0$. Then, for $n$ large enough, with probability larger than $1-3n^{-p} - e^{- \frac{n}{2M^4}\left ( (\delta - \frac{12kM^2}{\sqrt{n}}) \right )^2}$, $X_1, \hdots, X_n$ is $ \sqrt{p_{min} n}$-clusterable. Moreover, on the same event, we have
\begin{align*}
\| \hat{\c}_n - \ckm \| \leq \frac{60M}{n p_{min}^2}.
\end{align*}
\end{thm}
A proof of Theorem \ref{thm:margin_condition_gives_clusterability} can be found in Section \ref{sec:proof_thm_margin_condition_gives_clusterability}. Combining Theorem \ref{thm:margin_condition_gives_clusterability} and Theorem \ref{thm:classif_error_km} ensures that whenever $P$ satisfies a margin condition, then with high probability the classification error of the $k$-means codebook starting from a good initialization, $\hat{R}_{classif}(\ckm,\hat{\c}_n)$, is of order $1/(np_{min})$. Thus, according to Corollary \ref{cor:link_distortion_classification}, the classification error $\hat{R}_{classif}(\ckm,\c^*(\ckm))$ should be of order $\sqrt{(k+\log(|\bar{\mathcal{M}}|)/n}$, for $n$ large enough. This suggests that the misclassifications of $\ckm$ are mostly due to the misclassifications of $\hat{\c}_n$, rather than the possible difference between $\hat{\c}_n$ and $\ckm$.

Combining the bound on $\|\hat{\c}_n - \ckm\|$ with a bound on $\|\hat{\c}_n - \c^*(\hat{\c}_n)\|$ that may be deduced from Theorem \ref{thm:fast_rates_distortion} and Proposition \ref{prop:properties_under_MC} may lead to guarantees on the distortion and classification risk $R_{dist}(\ckm)$ and $R_{classif}(\ckm,\c^*(\ckm))$. An illustration of this point is given in Corollary \ref{cor:classification_error_GM}.
 
 Note also that the condition on the initialization in Theorem \ref{thm:classif_error_km}, that is $ g \leq f^2/128-1$, can be written as $g \leq n p_{min}/2-1$ in the framework of Theorem \ref{thm:margin_condition_gives_clusterability}. Thus, for $n $ large enough, provided that $R_{dist}(\c^*) >0$, every initialization $\c^{(0)}$ turns out to be a good initialization.   
 
 \begin{cor}\label{cor:classification_error_GM}
 Under the assumptions of Proposition \ref{prop:GM_margincondition}, for $k=2$, $\Sigma_i = \sigma I_d$, and $p_{min}=1/2$, if $n$ is large enough then 
 \[
 \mathbb{E} R_{classif}\left ( \mathcal{C}(\ckm), \mathcal{C}(\m) \right ) \leq C \sigma \sqrt{\frac{\log(n)}{n}},
 \]
 where $\ckm$ denotes the output of the Lloyd's algorithm.
 \end{cor} 
Note that in this case $\mathcal{C}(\m)$ corresponds to the Bayes classification $\mathcal{C}^*$. Thus, in the ``easy'' classification case $\frac{\sigma}{B}$ small enough, the output of the Lloyd's algorithm achieves the optimal classification error. It may be also worth remarking that this case is peculiar in the sense that $\mathcal{C}(\c^*) = \mathcal{C}(\m)$, that is the classification targeted by $k$-means is actually the optimal one. In full generality, since $\c^* \neq \m$, a bias term accounting for $R_{classif} \left ( \mathcal{C}(\c^*),\mathcal{C}(\m) \right )$ is likely to be incurred. 

\section{Conclusion}

        As emphasized by the last part of the paper, the margin condition we introduced seems a relevant assumption when $k$-means based procedures are used as a classification tool. Indeed, such an assumption in some sense postulates that there exists a natural classification that can be reached through the minimization of a least-square criterion. Besides, it also guarantees that both a true empirical distortion minimizer and the output of the Lloyd's algorithm approximate well this underlying classification.  
        
        From a technical point a view, this condition was shown to connect a risk in distortion and a risk in classification. As mentioned above, this assesses the relevance of trying to find a good classifier via minimizing a distortion, but this also entails that the distortion risk achieves a fast convergence rate of $1/n$. Though this rate seems optimal on the class of distributions satisfying a margin condition, a natural question is whether fast rates of convergence for the distortion can occur more generally. 
        
        In full generality, the answer is yes. Indeed, consider $P_0$ a two-component truncated Gaussian mixtures on $\mathbb{R}$ satisfying the requirements of Proposition \ref{prop:GM_margincondition}. Then set $P$ has a distribution over $\mathbb{R}^2$, invariant through rotations, and that has marginal distribution $P_0$ on the first coordinate. According to Corollary \ref{cor:optimal_codebook_invariance}, $P$ cannot satisfy a margin condition. However, by decomposing the distortion of codebooks into a radial and an orthogonal component, it can be shown that such a distribution gives a fast convergence rate for the expected distortion of the empirical distortion minimizer. 
        
        The immediate questions issued by Proposition \ref{prop:CN_optimality} and the above example are about the possible structure of the set of optimal codebooks: can we find distributions with infinite set of optimal codebooks that have finite isometry group? If not, through quotient-like operations can we always reach a fast convergence rate for the empirical risk minimizer? Beyond the raised interrogations, this short example allows to conclude that our margin condition cannot be necessary for the distortion of the ERM to converge fast. 
               
\section{Proofs}
\subsection{Proof of Proposition \ref{prop:GM_margincondition}}\label{subsec:proof_prop_GM_margincondition}
The proof of Proposition \ref{prop:GM_margincondition} is based on the following Lemma.
\begin{lem}{\cite[Lemma 4.2]{Levrard18}}\label{lem:Gaussiancalculus}
 Denote by $\eta = \sup_{j=1, \hdots, k}{1 - N_i}$. Then the risk $R(\m)$ may be bounded as follows.
 \begin{align}\label{eq:GM_meansrisk}
 R(\m) \leq \frac{\sigma^2 k \theta_{max} d }{(1-\eta)},
 \end{align}
 where $\theta_{max} = \max_{j=1, \hdots, k} {\theta_j}$. For any $0 < \tau <1/2$, let $\c$ be a codebook with a code point $c_i$ such that $\| c_i - m_j\| > \tau \tilde{B}$, for  every $j$ in $\{1, \hdots, k\}$. Then we have
 \begin{align}\label{eq:GM_risklowerbound}
 R(\c) > \frac{\tau^2 \tilde{B}^2 \theta_{min}}{4} \left ( 1 - \frac{2 \sigma \sqrt{d}}{\sqrt{2\pi } \tau \tilde{B}}e^{-\frac{\tau^2 \tilde{B}^2}{4 d \sigma^2}} \right )^d,
 \end{align}
 where $\theta_{min} = \min_{j=1, \hdots, k}{\theta_j}$. At last, if $\sigma^- \geq c_- \sigma$, for any $\tau'$ such that $2 \tau + \tau' < 1/2$, we have
 \begin{align}\label{eq:GM_gaussianweightfunction}
 \forall t \leq \tau' \tilde{B} \quad  p(t) \leq t \frac{2 k^2 \theta_{max} M^{d-1} S_{d-1}}{(2 \pi)^{d/2}(1- \eta) c_-^d \sigma^d} e^{- \frac{\left [ \frac{1}{2} - (2 \tau + \tau') \right ]^2 \tilde{B}^2}{2 \sigma^2}},  
 \end{align}
 where $S_{d-1}$ denotes the Lebesgue measure of the unit ball in $\mathbb{R}^{d-1}$.
 \end{lem}
 \begin{proof}{Proof of Proposition \ref{prop:GM_margincondition}}
 We let $\tau = \frac{c_1 \sqrt{d}\sigma}{\tilde{B}}$, with $c_1=\sqrt{\frac{k 2^{d+2}}{(1-\eta)\theta_{min}}}$. Note that $\frac{\sigma}{\tilde{B}} \leq \frac{1}{16\sqrt{d} c_1}$ entails $\tau \leq \frac{1}{16}$. Let $\c$ be a codebook with a code point $c_i$ such that $\| c_i - m_j\| > \tau \tilde{B}$, for  every $j$ in $\{1, \hdots, k\}$. Then \eqref{eq:GM_risklowerbound} gives
 \begin{align*}
 R(\c)&> \frac{c_1^2 \sigma ^2\theta_{min} d 2^{-d}}{4} \\
      &> \frac{k \sigma^2 d}{(1-\eta)} \\
      &> R(\m),
 \end{align*}
 according to \eqref{eq:GM_meansrisk}. Thus, an optimal codebook $\c^*$ satisfies, for all $j=1, \hdots, k$, $\|c_j^*-m_j\| \leq c_1 \sqrt{d} \sigma$, up to relabeling. Under the condition $\frac{\sigma}{\tilde{B}} \leq \frac{1}{16\sqrt{d} c_1}$ and $\tau = \frac{c_1 \sqrt{d}\sigma}{\tilde{B}}$, we have, since $\tau \leq \frac{1}{16}$, for every $\c^* \in \mathcal{M}$ and $j=1, \hdots, k$,
 \begin{align*}
 \tilde{B}  \geq \frac{B}{2}, \quad \mbox{and} \quad 
 \mathcal{B} \left ( m_j,\frac{\tilde{B}}{4} \right ) \subset V_j(\c^*). 
 \end{align*}
 We thus deduce that
 \begin{align*}
 p_{min} & \geq \frac{\theta_{min}}{(2\pi)^{\frac{d}{2}}} \int_{\mathcal{B}(0,\frac{\tilde{B}}{4})} {e^{-\frac{\|u\|^2}{2}}du} \\
          & \geq \frac{\theta_{min}}{\frac{d}{2}} \left ( 1 - \frac{4 \sigma \sqrt{d}}{\sqrt{2 \pi} \tilde{B}} e^{-\frac{\tilde{B}^2}{16 d \sigma ^2}} \right )^{d} \\
          & \geq \frac{\theta_{min}}{2^d (2\pi)^{\frac{d}{2}}}. 
 \end{align*}
 Recall that we have $M \leq c \sigma $ for some constant $c>0$, and $\sigma_- \geq c_- \sigma$. If $\tilde{B} / \sigma $ additionally satisfies $\frac{\tilde{B}^2}{ \sigma^2} \geq 32 \log \left (\frac{2^{d+5} S_{d-1} k^2 c^{d+1}}{(1-\eta) \theta_{min} c_{-}^d} \right )$, choosing $\tau' = \frac{1}{8}$ in \eqref{eq:GM_gaussianweightfunction} leads to, for $t \leq \frac{\tilde{B}}{8}$,
 \begin{align*}
 p(t) & \leq  t \frac{2 k^2 M^{d-1} S_{d-1}}{(2 \pi)^{\frac{d}{2}}(1-\eta) c_-^d \sigma^d} e^{-\frac{\tilde{B}^2}{32 \sigma^2}} \\
      &  \leq t \frac{\theta_{min}M^{d-1}}{2^{d+4}c^{d+1} \sigma^{d} (2 \pi)^{\frac{d}{2}}} \\  
       & \leq t \frac{\tilde{B} \theta_{min}}{(2 \pi)^{\frac{d}{2}}2^{d+8}M^2} \leq \frac{B p_{min}}{128M^2}.
 \end{align*}
 Hence $P$ satisfies a margin condition with radius $\tilde{B}/8$. Note that according to Proposition \ref{prop:properties_under_MC}, no local minimizer of the distortion may be found in $\mathcal{B}(\c^*,r)$, for $\c^* \in \mathcal{M}$ and $r = \frac{Br_0}{4\sqrt{2M}}$. Note that $r \geq \frac{\tilde{B}^2}{64\sqrt{2}c \sigma}$ and $\|\c^* - \m \| \leq c_1 \sigma \sqrt{kd}$. Thus, if $ \frac{\sigma^2}{\tilde{B}^2} \leq \frac{1}{128\sqrt{2}c_1 c \sqrt{kd}}$, $\c^*$ is unique (up to relabeling).
 \end{proof}
\subsection{Proof of Proposition \ref{prop:CN_optimality}}\label{subsec:proof_CN_optimality} 
\begin{proof}
Let $0\leq t<\frac{1}{2}$, $\c^* \in \mathcal{M}$, and for short denote by $r_{ij}$, $V_i$, $p_i$ the quantities $\|c_i^* - c_j^*\|$, $V_i(\c^*)$ and $p_i(\c^*)$. Also denote by $u_{ij}$ the unit vector $\frac{c_j^* - c_i^*}{r_{ij}}$, $c^t_i = c_i^* + 2t(c_j^*-c_i^*)$, and by $H^t_{ij} = \left \{ x | \quad  \| x- c^t_i \| \leq \|x-c_j^*\| \right \}$. We design the quantizer $Q^t_i$ as follows: for every $\ell \neq i,j$, $Q^t(V_\ell) = c_\ell^*$, $Q^t((V_i \cup V_j) \cap H^t_{ij}) = c^t_i$, and $Q^t((V_i \cup V_j) \cap (H^t_{ij})^c)=c_j^*$.
Then we may write
\begin{align}\label{eq:localperturbation_distortion_1}
0 \leq R_{dist}(Q^t_i) - R_{dist}(\c^*) = 4 p_i r_{ij}^2 t^2 + P \left ( (\|x-c_i^t\|^2 - \|x-c_j^*\|^2) \mathbbm{1}_{V_j \cap H^t_{ij}}(x) \right ).
\end{align}
On the other hand, straightforward calculation show that 
$V_j \cap H^t_{ij} = \left \{ x | \quad 0 \leq  \left\langle x - \frac{c_i^* + c_j^*}{2}, u_{ij} \right\rangle \leq tr_{ij} \right \}$. Besides, for any $x \in V_j \cap H^t_{ij}$, denoting by $s$ the quantity $ \left\langle x - \frac{c^*_i+c^*_j}{2},u_{ij} \right\rangle$, we have
\begin{align*}
\|x-c_i^t\|^2 - \|x-c_j^*\|^2 & = 2 \left \langle (1-2t)(c_j^*-c_i^*), x - \frac{c_i^*+c_j^*}{2} - t(c_j^*-c_i^*)\right\rangle \\ 
 & = 2 \left [ r_{ij} s(1-2t) - t(1-2t)r_{ij}^2 \right ] \\
 & =2r_{ij}(1-2t)(s-tr_{ij}). 
\end{align*}
Thus \eqref{eq:localperturbation_distortion_1} may be written as
\begin{align*}
(1-2t)\int_{0}^{tr_{ij}}(tr_{ij}-s)dp_{ij}(s) \leq 2 p_i r_{ij} t^2.
\end{align*}
Integrating by parts leads to $\int_{0}^{tr_{ij}}(tr_{ij}-s)dp_{ij}(s) = \int_{0}^{tr_{ij}}p_{ij}(u)du$. Thus
\[
\int_{0}^{tr_{ij}}p_{ij}(\c^*,s)ds \leq 2 t^2 r_{ij}(\c^*) \frac{p_i(\c^*)}{1-2t}.
\]
The other inequalities follows from the same calculation, with the quantizer moving $c_i^*$ to $c_i^* - 2t(c_j^* - c_i^*)$, and the quantizer moving $c_i^*$ and $c_j^*$ to $c_i^* + t(c_j^*-c_i^*)$ and $c_j^*+t(c_j^*-c_i)^*$, leaving the other cells $V_\ell$ unchanged. 
\end{proof}
\subsection{Proof of Corollary \ref{cor:link_distortion_classification}}\label{subsec:proof_cor_link_distortion_classification}
\begin{proof}
According to \cite[Lemma 4.4]{Levrard15}, if $R_{dist}(\c) - R_{dist}(\c^*) \leq \delta$, then $\| \c - \c^*(\c) \| \leq r$, with $r= \frac{Br_0}{4 \sqrt{2}M}$. We may decompose the classification error as follows.
\begin{align*}
R_{classif} \left( \mathcal{C}(\c),\mathcal{C}(\c^*(\c)) \right ) = P \left ( \bigcup_{j \neq i} V_j(\c^*) \cap V_i(\c) \right ). 
\end{align*}
According to \cite[Lemma 4.2]{Levrard15}, 
\[
\bigcup_{j \neq i} V_j(\c^*) \cap V_i(\c) \subset \mathcal{B} \left ( N(\c^*(\c)), \frac{4 \sqrt{2} M}{B} \| \c - \c^*(\c) \| \right ).
\]
Thus, since $P$ satisfies a margin condition with radius $r_0$,
\begin{align*}
R_{classif} \left( \mathcal{C}(\c),\mathcal{C}(\c^*(\c)) \right ) & \leq \frac{4\sqrt{2}p_{min}}{128M} \| \c - \c^*(\c) \| \\
& \leq \frac{\sqrt{p_{min}}}{16M} \sqrt{R_{dist}(\c) - R_{dist}(\c^*)}, 
\end{align*}
according to Proposition \ref{prop:properties_under_MC}.
\end{proof}
\subsection{Proof of Theorem \ref{thm:fast_rates_distortion}}\label{subsec:proof_thm_fast_rates_distortion}
The proof of Theorem \ref{thm:fast_rates_distortion} relies on the techniques developed in the proof of \cite[Theorem 3.1]{Levrard15} and the following result from \cite{Biau08}.
\begin{thm}{\cite[Corollary 2.1]{Biau08}}\label{thm:slow_rates_gerard}
Assume that $P$ is $M$-bounded. Then, for any $x>0$, we have
\begin{align*}
R_{dist}(\hat{\c}_n) - R_{dist}(\c^*) \leq \frac{12kM^2 +M^2 \sqrt{2x} }{\sqrt{n}},
\end{align*}
with probability larger than $1-e^{-x}$.
\end{thm}
We are now in position to prove Theorem \ref{thm:fast_rates_distortion}.
\begin{proof}{Proof of Theorem \ref{thm:fast_rates_distortion}}
Assume that $P$ satisfies a margin condition with radius $r_0$, and denote by $r = \frac{Br_0}{4\sqrt{2}M}$, $\delta = \frac{p_{min}}{2}r^2 \wedge \varepsilon$, where $\varepsilon$ denotes the separation factor in Definition \ref{def:separation_factor}. For short denote, for any codebook $\c \in (\mathbb{R^d})^k$, by $\ell(\c,\c^*) = R_{dist}(\c) - R_{dist}(\c^*)$. According to \cite[Lemma 4.4]{Levrard15}, if $\| \c - \c^*(\c)\| \geq r$, then $\ell(\c,\c^*) \geq \frac{p_{min}}{2}r^2 \wedge \varepsilon$. Hence, if $\ell(\c,\c^*) < \delta$, $\|\c- \c^*(\c)\| <r$. 

Using Theorem \ref{thm:slow_rates_gerard}, we may write
\begin{align}\label{eq:majoration_proba_outofbasin}
\mathbb{P} \left ( \ell(\hat{\c}_n,\c^*) > \delta \right ) \leq e^{- \frac{n}{2M^4}\left ( (\delta - \frac{12kM^2}{\sqrt{n}}) \right )^2}.
\end{align} 
Now, for any $x>0$ and constant $C$ we have
\begin{align*}
\mathbb{P}&\left [ \left ( \ell(\hat{\c}_n,\c^*) > C \frac{2}{p_{min}} \frac{\left ( k + \log \left ( \left | \bar{\mathcal{M}} \right | \right ) \right )M^2}{n} + \frac{288 M^2}{p_{min}n}x + \frac{64M^2}{n}x \right )\cap (\ell(\hat{\c}_n,\c^*) < \delta) \right ] \\
    & \leq \mathbb{P}\left [ \left ( \ell(\hat{\c}_n,\c^*) > C \frac{2}{p_{min}} \frac{\left ( k + \log \left ( \left | \bar{\mathcal{M}} \right | \right ) \right )M^2}{n} + \frac{288 M^2}{p_{min}n}x + \frac{64M^2}{n}x \right )\cap \left (\hat{\c}_n \in \mathcal{B}(\mathcal{M},r) \right ) \right ].
\end{align*}
Proceeding as in the proof of \cite[Theorem 3.1]{Levrard15} entails, for every $x>0$,
\begin{align}\label{eq:deviation_proba_in_basin}
\mathbb{P}\left [ \left ( \ell(\hat{\c}_n,\c^*) > C \frac{2}{p_{min}} \frac{\left ( k + \log \left ( \left | \bar{\mathcal{M}} \right | \right ) \right )M^2}{n} + \frac{288 M^2}{p_{min}n}x + \frac{64M^2}{n}x \right )\cap \left (\hat{\c}_n \in \mathcal{B}(\mathcal{M},r) \right ) \right ] \leq e^{-x},
\end{align}
for some constant $C>0$. Note that \eqref{eq:majoration_proba_outofbasin} and \eqref{eq:deviation_proba_in_basin} are enough to give a deviation bound in probability. For the bound in expectation, 
set $\beta = \frac{2 C\left ( k + \log \left ( \left | \bar{\mathcal{M}} \right | \right ) \right )M^2 }{np_{min}}$.
On one hand, Theorem \ref{thm:slow_rates_gerard} and \eqref{eq:majoration_proba_outofbasin} yield that
         \begin{align*}
         \mathbb{E}(\ell(\hat{\c}_n,\c^*) \mathbbm{1}_{\ell(\hat{\c}_n,\c^*)>\delta}) & \leq  \int_{\delta}^{\infty} \mathbb{P}(\ell(\hat{\c}_n,\c^*)>u)du \\
         & \leq \begin{multlined}[t] \left [ \frac{12kM^2}{\sqrt{n}} - \delta  + \int_{\frac{12kM^2}{\sqrt{n}}}^{\infty} \mathbb{P}(\ell(\hat{\c}_n,\c^*)>u) du \right ] \mathbbm{1}_{\delta < \frac{12kM^2}{\sqrt{n}}} \\
         + \left [\int_{0}^{\infty} \mathbb{P}(\ell(\hat{\c}_n,\c^*) - \frac{12kM^2}{\sqrt{n}} \geq (\delta - \frac{12kM^2}{\sqrt{n}}) + u)du \right ] \mathbbm{1}_{\delta \geq \frac{12kM^2}{\sqrt{*n}}}
         \end{multlined}\\
         & \leq \left [ \frac{20 k M^2}{\sqrt{n}} - \delta \right ]\mathbbm{1}_{\delta < \frac{12kM^2}{\sqrt{n}}} + \left [\int_{0}^{\infty} e^{- \frac{n}{2M^4}\left ( (\delta - \frac{12kM^2}{\sqrt{n}}) + u \right )^2} du \right ]\mathbbm{1}_{\delta \geq \frac{12kM^2}{\sqrt{*n}}}   \\
         & \leq  \left [ \frac{20 k M^2}{\sqrt{n}} - \delta \right ] \mathbbm{1}_{\delta < \frac{12kM^2}{\sqrt{n}}} + \left [ e^{- \frac{n}{2M^4}\left ( (\delta - \frac{12kM^2}{\sqrt{n}}) \right )^2} \frac{M^2}{\sqrt{n}} \right ] \mathbbm{1}_{\delta \geq \frac{12kM^2}{\sqrt{n}}}, 
         \end{align*}
         where we used $\sqrt{\pi} \leq 2$ and $(a+b)^2 \geq a^2 + b^2$ whenever $a$,$b \geq0$.
On the other hand, \eqref{eq:deviation_proba_in_basin} entails
         \begin{align*}
         \mathbb{E}(\ell(\hat{\c}_n,\c^*) \mathbbm{1}_{\ell(\hat{\c}_n,\c^*)\leq \delta}) & \leq (\beta - \delta) \mathbbm{1}_{\delta < \beta} + \left [ \beta + \int_{\beta}^{\infty}\mathbb{P}((\ell(\hat{\c}_n,\c^*) \geq u) \cap \left (\hat{\c}_n \in \mathcal{B}(\mathcal{M},r) \right )) du \right ] \mathbbm{1}_{\delta \leq \beta} \\
        & \leq \beta + \frac{252 M^2}{n p_{min}}, 
         \end{align*}
         where we used $p_{min} \leq 1$. Collecting the pieces gives the result of Theorem \ref{thm:fast_rates_distortion}.
\end{proof}
\subsection{Proof of Proposition \ref{prop:minimax_new}}\label{subsec:minimax_proofs}
\begin{proof}
Assume that $dim(\mathcal{H})=d$, and let $z_1, \hdots, z_k$ be in $\mathcal{B}(0,M-\Delta/8)$ such that $\|z_i - z_j\| \geq \Delta$, and $\Delta \leq 2M$. Then slightly anticipating we may choose
\[
\Delta \leq \frac{3M}{4k^{1/d}}.
\]
Let $\rho = \Delta/8$, and for $\sigma \in \left \{-1,1 \right \}^k$ and $\delta \leq 1$ denote by $P_{\sigma}$ the following distribution. For any $A \subset \mathcal{H}$, and $i = 1, \hdots, k$,
\[
P_{\sigma}(A \cap \mathcal{B}(z_i,\rho)) = \frac{1}{2 \rho k} \left [ (1+ \sigma_i \delta) \lambda_1 (e_1^*(A-z_i) \cap [0,\rho] + (1-\sigma_i \delta) \lambda_1(e_1^*(A-z_i) \cap [-\rho,0] \right ],
\]
where $e_1^*$ denotes the projection onto the first coordinate and $\lambda_1$ denote the $1$-dimensional Lebesgue measure. Note that for every $i$, $P_{\sigma}(\mathcal{B}(z_i,\rho))=1/k$. We let $\c_\sigma$ denote the codebook whose codepoints are $c_{\sigma,i} = z_i + \sigma_i \delta/2$. For such distributions $P_\sigma$'s, it is shown in Section \ref{subsec:proof_minimaxnew_intermediateresults} that
               \[
               \left \{
               \begin{array}{@{}ccc}
               p_{min} &=& \frac{1}{k}, \\
               B &\geq& \frac{3\Delta}{4}, \\
               r_0 & \geq & \frac{\Delta}{4}, \\
               \varepsilon & \geq &  \frac{ \Delta^2}{96k}.
               \end{array}
               \right .
               \]
               Half of the proof of Proposition \ref{prop:minimax_new} is based on the following Lemma. For simplicity, we write $R(\hat{\c},P_{\sigma})$ for the distortion of the codebook $\hat{\c}$ when the distribution is $P_\sigma$.
\begin{lem}\label{lem:minimax_new_risks}
     For every $\sigma$, $\sigma'$ in $\{-1,+1\}^k$, 
     \[
     R(\c_{\sigma'},P_\sigma) - R(\c_{\sigma},P_{\sigma}) = \frac{2\delta^2 \rho^2}{k} H(\sigma,\sigma') = \frac{2}{k} \|\c_{\sigma} - \c_{\sigma}'\|^2,
     \]
     where $H(\sigma,\sigma') = \sum_{i=1}^{k}|\sigma_i - \sigma'_i|/2$. Moreover, for every codebook $\hat{\c}$ there exist $\hat{\sigma}$ such that, for all $\sigma$, 
     \[
     R(\hat{\c},P_\sigma) - R(\c_\sigma,P_\sigma) \geq \frac{1}{4k} \| \c_{\hat{\sigma}} - \c_\sigma \|^2.
     \]
\end{lem}

Lemma \ref{lem:minimax_new_risks}, whose proof is to be found in Section \ref{subsec:proof_minimaxnew_intermediateresults}, ensures that our distortion estimation problem boils down to a $\sigma$ estimation problem. Namely, we may deduce that 
\begin{align*}
\inf_{\hat{Q}} \sup_\sigma \mathbb{E} (R(\hat{Q},P_\sigma)-R(\c_\sigma,P_\sigma)) & \geq  \frac{\delta^2 \rho^2}{4k} \inf_{\hat{\sigma}} \sup_{\sigma} H(\hat{\sigma},\sigma)).
\end{align*}
The last part of the proof derives from the following.
\begin{lem}\label{lem:minimax_Hellinger}
If $k \geq n$ and $\delta \leq \sqrt{k}/2\sqrt{n}$, then, for every $\sigma$ and $\sigma'$ such that $H(\sigma,\sigma')=1$, 
\[
h^2(P_\sigma^{\otimes n}, P_{\sigma'}^{\otimes n}) \leq 1/4,
\]
where $h^2$ denotes the Hellinger distance.
\end{lem} 
Thus, recalling that $\Delta = {3M}/(4k^{1/d})$ and $\rho=\Delta/8$, if we choose $\delta = \frac{\sqrt{k}}{2\sqrt{n}}$, a direct application of \cite[Theorem 2.12]{Tsybakov08} yields
\[
\inf_{\hat{Q}} \sup_\sigma \mathbb{E} (R(\hat{Q},P_\sigma)-R(\c_\sigma,P_\sigma)) \geq \frac{9}{2^{16}}M^2 \frac{k^{1-\frac{2}{d}}}{n}.
\]
\end{proof}
\subsection{Intermediate results for Section \ref{subsec:minimax_proofs}}\label{subsec:proof_minimaxnew_intermediateresults}
First we prove Lemma \ref{lem:minimax_new_risks}.
\begin{proof}{Proof of Lemma \ref{lem:minimax_new_risks}}
We let $I_i$ denote the $1$ dimensional interval $[z_i - \rho e_1, z_i + \rho e_1]$, and $V_i$ the Voronoi cell associated with $z_i$. At last, for a quantizer $Q$ we denote by $R_i(Q,P_\sigma)$ the contribution of $I_i$ to the distortion, namely $R_i(Q,P_\sigma) = P_\sigma \|x - Q(x)\|^2 \mathbbm{1}_{V_i}(x) = P_\sigma \|x - Q(x)\|^2 \mathbbm{1}_{I_i}(x)$. Since ${\Delta}/2 - 3 \rho >0$, $I_i \subset V_i(\c_{\sigma})$, for every $i$ and $\sigma$.  According to  the centroid condition (Proposition \ref{prop:centroid_condition}), if $|Q(I_i)|=1$, that is only one codepoint is associated with $I_i$, then 
\begin{align}\label{eq:minimax_local_convexity}
R(Q,P_{\sigma}) & = R(\c_\sigma,P_\sigma)  + \sum_{i=1}^{k} P_{\sigma}(I_i) \| Q(I_i) - c_{\sigma,i}\|^2,
\end{align}
hence the first part of Lemma \ref{lem:minimax_new_risks}, with $Q$ associated to $\c_{\sigma'}$. 

Now let $\c$ be a codebook, and denote by $Q$ the associated quantizer. Denote by $n_i = |Q(I_i)|$, $n_i^{in} = |Q(I_i) \cap V_i |$ and $n_i^{out} = |Q(I_i) \cap V_i^c$. If $n_i^{out} \geq 1$, then there exists $x_0 \in I_i$ such that $\|Q(x_0) - x_0\| \geq \Delta/2 - \rho$. Then, for any $x \in I_i$ it holds $\|Q(x) - x\| \geq \| Q(x) - x_0 \| - 2\rho \geq \Delta/2 - 3 \rho$. We deduce that for such an $i$, and every $\sigma$, 
\[
R_i(Q,\sigma) \geq \frac{1}{k}\left \| \frac{\Delta}{2} - 3 \rho ^2 \right \| = \frac{\rho^2}{k}.
\]
The second base inequality is that, for every $Q$ such that $Q(I_i) = z_i$, and every $\sigma$,
\[
R_i(Q,\sigma) = \frac{\rho^2}{3k}.
\]
We are now in position to build a new quantizer $\tilde{Q}$ that outperforms $Q$. 
\begin{itemize}
\item If $n_i^{in}=1$ and $n_i^{out}=0$, then $\tilde{Q}(I_i) = \pi_{I_i}(Q(I_i))$, where $\pi_{I_i}$ denote the projection onto $I_i$.
\item If $n_i^{out}\geq1$, then $\tilde{Q}(I_i)=z_i$.
\item If $n_i^{in} \geq 2$ and $n_i^{out}=0$, then $\tilde{Q}(I_i)=z_i$.
\end{itemize} 
Such a procedure defines a $k$-point quantizer $\tilde{Q}$ that sends every $I_i$ onto $I_i$. Moreover, we may write, for every $\sigma$
\begin{align*}
R(Q,P_{\sigma}) & =\sum_{n_i^{in}=1, n_i^{out}=0} R_i(Q,P_\sigma) + \sum_{n_i^{out} \geq 1} R_i(Q,P_\sigma) + \sum_{n_i^{out}=0, n_i^{in} \geq 2} R_i(Q,P_\sigma) \\
               & \geq \sum_i R_i(\tilde{Q},P_\sigma) + \left | \{i|n_i^{out}\geq1 \} \right | \frac{2\rho^2}{3k} - \left | \{ i | n_i^{out}=0,n_i^{in} \geq 2 \} \right | \frac{\rho^2}{3k}.    
\end{align*}
Since $\left | \{i|n_i^{out}\geq 1 \} \right | \geq \left | \{ i | n_i^{out}=0,n_i^{in} \geq 2 \} \right |$, we have $R(Q,P_\sigma) \geq R(\tilde{Q},P_\sigma)$, for every $\sigma$. Note that such a quantizer $\tilde{Q}$ is indeed a nearest-neighbor quantizer, with images $\tilde{c}_i \in I_i$. For such a quantizer $\tilde{\c}$, \eqref{eq:minimax_local_convexity} yields, for every $\sigma$,
\[
R(\tilde{\c},P_\sigma) - R(\c_\sigma,P_\sigma) = \frac{\| \tilde{\c} - \c_\sigma \|^2}{k}. 
\]
Now, if $\c_{\hat{\sigma}}$ denotes $\arg\min_{\c_\sigma}\|\c_\sigma - \tilde{\c}\|$, then, for every $\sigma$ we have
\[
\|\tilde{\c}-\c_\sigma\| \geq \frac{\|\c_{\hat{\sigma}} - \c_\sigma\|}{2}.
\]
Thus, recalling our initial codebook $\c$, for every $\sigma$, $R(\c,P_\sigma) - R(\c_{{\sigma}},P_\sigma) \geq \frac{1}{4k}\| \c_{\hat{\sigma}} - \c_\sigma \|^2$. 
\end{proof}
\subsection{Proof of Proposition \ref{prop:connections_B_epsilon_r0}}\label{subsec:proof_prop_connections_B_epsilon_r0}
Let $\c^* \in \mathcal{M}$ and $i \neq j$ such that $\|c_i^* - c_j^*\|=B$. We denote by $Q_{i,j}$ the $(k-1)$- points quantizer that maps $V_{\ell}(\c^*)$ onto $c_\ell^*$, for $\ell \neq i,j$, and $V_i(\c^*) \cup V_j(\c^*)$ onto $\frac{c_i^* + c_j^*}{2}$. Then $R_{dist}(Q_{i,j}) - R_{dist}(\c^*) = (p_i(\c^*) + p_j(\c^*))\frac{B^2}{4} \leq \frac{B^2}{4}$. Thus, denoting by $\c^{*,(k-1)}$ an optimal $(k-1)$-points quantizer, $R_{dist}(\c^{*,(k-1)}) - R_{dist}(\c^*) \leq \frac{B^2}{4}$. Since $\varepsilon \leq R_{dist}(\c^{*,(k-1)}) - R_{dist}(\c^*)$, the first part of Proposition \ref{prop:connections_B_epsilon_r0} follows.

For the same optimal codebook $\c^*$, we denote for short by $p(t)$ the quantity
\[
p(t) = P \left ( \left \{ x | \quad 0 \leq \left\langle x - \frac{c_i + c_j}{2}, \frac{c_i-c_j}{r_{i,j}(\c^*)} \right\rangle \leq t \right \} \cap V_i(\c) \right ),
\]
and by $p_i = p_i(\c^*)$. According to Proposition \ref{prop:centroid_condition}, we have
\begin{align}\label{eq:centroid_projected_B}
p_i \frac{B}{2} &= P \left ( \left\langle x - \frac{c_i^*+c_j^*}{2}, \frac{c_i^*-c_j^*}{r_{i,j}} \right\rangle \mathbbm{1}_{V_i(\c^*)} (x) \right ) \notag \\
                & = \int_{0}^{2M} t dp(t).
\end{align}
Assume that $r_0 > B$. Then 
\begin{align*}
p(B) & \leq \frac{B p_{min}}{128 M^2} B \\
     & \leq \frac{p_{min}}{32}.
\end{align*}
On the other hand, \eqref{eq:centroid_projected_B} also yields that
\begin{align*}
p_i \frac{B}{2} & \geq  \int_{B}^{2M} tdp(t) \\
                & \geq B \left ( p_i - p(B) \right ) \\
                & \geq p_i B \frac{31}{32}, 
\end{align*}
hence the contradiction.
\subsection{Proof of Proposition \ref{prop:connection_clusterability_empiricalmargin}}\label{sec:Proof of Proposition prop:connection_clusterability_empiricalmargin}
  The proof of Proposition \ref{prop:connection_clusterability_empiricalmargin} is based on the following Lemma, that connects the clusterability assumption introduced in Definition \ref{def:clusterability} to another clusterability definition introduced in \cite{Kumar10}.
  \begin{lem}{\cite[Lemma 10]{Tang16b}}\label{lem:connection_clusterability1_clusterabilityKK}
  Assume that there exist $d_{rs}$'s, $r \neq s$, such that, for any $r \neq s$ and $x \in V_s(\hat{\c}_n)$,
  \[
  P_n \left ( \left \{ x | \quad \|x_{rs} - \hat{c}_r\| \leq \|x_{rs} - \hat{c}_s \| + d_{rs} \right \} \right ) < \hat{p}_{min},
  \]
  where $x_{rs}$ denotes the projection of $x$ onto the line joining $\hat{c}_r$ and $\hat{c}_s$. Then, for all $r \neq s$,
  \[
  \|\hat{c}_r - \hat{c}_s\| \geq d_{rs}.
  \] 
  \end{lem}
\begin{proof}{Proof of Proposition \ref{prop:connection_clusterability_empiricalmargin}}  
Now let  $x \in \left \{ x | \quad \|x_{rs} - \hat{c}_r\| \leq \|x_{rs} - \hat{c}_s \| + d_{rs} \right \} \cap V_s(\hat{\c}_n)$, for $d_{rs} \leq 2M$. Then
\begin{align*}
\|x-\hat{c}_r\| & \leq \|x - \hat{c}_s\| + d_{rs} \\
\|x - \hat{c}_s\| & \leq \|x-\hat{c}_r\|.
\end{align*} 
Taking squares of both inequalities leads to
\begin{align*}
\left\langle \hat{c}_s - \hat{c}_r, x - \frac{\hat{c}_r + \hat{c}_s}{2} \right\rangle & \geq 0 \\
 2 \left\langle \hat{c}_s - \hat{c}_r, x - \frac{\hat{c}_r + \hat{c}_s}{2} \right\rangle & \leq d_{rs}^2 + 2d_{rs}\|x-\hat{c}_r\| \leq 8M d_{rs}.
\end{align*}
We deduce from above that $d(x,\partial V_{s}(\hat{\c}_n)) \leq \frac{8M}{\hat{B}}d_{rs}$, hence $x \in \mathcal{B}(N(\hat{\c}_n),\frac{8M}{\hat{B}}d_{rs})$. Set 
\[
d_{rs} = \frac{2Mf}{\sqrt{n_{min}}} \geq  f \sqrt{\hat{R}_{dist}(\hat{\c}_n)} \left ( \frac{1}{\sqrt{n_r}} + \frac{1}{\sqrt{n_s}} \right ),
\]
and assume that $\hat{p}\left( \frac{16M^2 f}{\sqrt{n \hat{p}_{min}}\hat{B}} \right ) \leq \hat{p}_{min}$. Then Lemma \ref{lem:connection_clusterability1_clusterabilityKK} entails that for all $\hat{\c}_n$ minimizing $\hat{R}_{dist}$ and $r \neq s$, $\|\hat{c}_r - \hat{c}_s\| \geq d_{rs}$. Hence $X_1, \hdots, X_n$ is $f$-clusterable. 
\end{proof}
\subsection{Proof of Theorem \ref{thm:margin_condition_gives_clusterability}}\label{sec:proof_thm_margin_condition_gives_clusterability}
\begin{proof}
Assume that $P$ satisfies a margin condition with radius $r_0$. For short we denote $R_{dist}(\c) - R_{dist}(\c^*)$ by $\ell(\c,\c^*)$. As in the proof of Theorem \ref{thm:fast_rates_distortion}, according to \eqref{eq:majoration_proba_outofbasin},  \eqref{eq:deviation_proba_in_basin}, choosing $x = p \log(n)$, for $n$ large enough, it holds, for every minimizer $\hat{\c}_n$ of $\hat{R}_{dist}$,
\begin{align*}
\ell(\hat{\c}_n,\c^*) & \leq C \frac{M^2p\log(n)}{np_{min}} \\
\ell(\hat{\c}_n,\c^*) & \geq \frac{p_{min}}{2}\| \hat{\c}_n - \c^*(\hat{\c}_n)\|^2,
\end{align*}
with probability larger than $1-n^{-p}-e^{- \frac{n}{32M^4}\left ( (\delta - \frac{12kM^2}{\sqrt{n}}) \right )^2}$. On this probability event we may thus write
\begin{align}\label{eq:hatc_close_c*}
\| \hat{\c}_n - \c^*(\hat{\c}_n)\| \leq C M \frac{\sqrt{p\log(n)}}{p_{min}\sqrt{n}}.
\end{align}
Since $N(\hat{\c}_n) \subset \mathcal{B}(N(\c^*(\hat{\c}_n),\sqrt{2}\|\hat{\c}_n - \c^*(\hat{\c}_n)\|)$, we get
\begin{align}\label{eq:hatp_first bound} 
\hat{p}(t) & \leq p\left (t + \sqrt{2}C M \frac{\sqrt{p\log(n)}}{p_{min}\sqrt{n}} \right ) \\ \notag
 & \leq \frac{B p_{min} t}{128 M^2 } + C \frac{B \sqrt{p \log(n)}}{M \sqrt{n}}, 
\end{align}
when $n$ is large enough so that $r_n < r_0$ and for $t \leq r_0-r_n$, with $r_n=C M \frac{\sqrt{p\log(n)}}{p_{min}\sqrt{n}}$. It remains to connect $\hat{p}_{min}$ and $\hat{B}$ with their deterministic counterparts. First, it is straightforward that
\begin{align}\label{eq:hatB_B}
\hat{B} \geq B - \sqrt{2}r_n \geq \frac{B}{2},
\end{align}
for $n$ large enough. The bound for $\hat{p}_{min}$ is slightly more involved. Let $i$ and $\hat{\c}_n$ such that $\hat{p}_{min} = P_n \left ( V_i(\hat{\c}_n) \right )$. Then we may write
\begin{align*}
\hat{p}_{min} & = P_n \left ( V_i(\hat{\c}_n) \right ) \\
              & = P_n \left ( V_i(\c^*(\hat{\c}_n)) \right )- P_n \left ( V_i(\c^*(\hat{\c}_n)) \cap V_i(\hat{\c}_n)^c \right ) + P_n \left ( V_i(\c^*(\hat{\c}_n))^c \cap V_i(\hat{\c}_n) \right ) .
\end{align*}
According to \cite[Lemma 4.2]{Levrard15}, $V_i(\c^*(\hat{\c}_n)) \Delta V_i(\hat{\c}_n) \subset \mathcal{B} \left ( N(\c^*(\hat{\c}_n), \frac{4\sqrt{2}M}{B}r_n ) \right )$, where $\Delta$ denotes the symmetric difference. Hoeffding's inequality gives
\begin{align*}
 \left| (P_n-P) \bigcup_{ \c^* \in \mathcal{M} } N(\c^*,\frac{4\sqrt{2}M}{B}r_n) \right | \leq \sqrt{\frac{2 p \log(n)}{n}},
\end{align*}
with probability larger than $1-n^{-p}$. Hence
\begin{align*}
P_n \left ( V_i(\c^*(\hat{\c}_n)) \Delta V_i(\hat{\c}_n) \right ) & \leq p \left ( \frac{4\sqrt{2}M}{B}r_n \right ) + \sqrt{\frac{2 p \log(n)}{n}} \\
& \leq C \frac{\sqrt{p \log(n)}}{\sqrt{n}}, 
\end{align*}
for $n$ large enough so that $\frac{4\sqrt{2}M}{B}r_n \leq r_0$. Concerning $P_n \left ( V_i(\c^*(\hat{\c}_n)) \right )$, using Hoeffding's inequality again we may write
\begin{align*}
P_n \left ( V_i(\c^*(\hat{\c}_n)) \right ) &\geq p_{min} - \sup_{\c^* \in \bar{\mathcal{M}},i =1, \hdots, k} \left |(P_n-P) V_i(\c^*)\right |  \\
 & \geq p_{min} - \sqrt{\frac{2 (p\log(n) + \log(k |\bar{\mathcal{M}}|)}{n}},
\end{align*}
with probability larger than $1-n^{-p}$. We deduce that
\[
\hat{p}_{min} \geq p_{min} - C \frac{\sqrt{p \log(n)}}{\sqrt{n}} \geq \frac{p_{min}}{2},
\]
for $n$ large enough. Thus, \eqref{eq:hatp_first bound} gives
\[
\hat{p}(t) \leq \frac{\hat{B}\hat{p}_{min}}{32M^2}t + C \frac{\hat{B} \sqrt{p \log(n)}}{M \sqrt{n}}.
\]
For $n$ large enough so that  $C \frac{\hat{B} \sqrt{p \log(n)}}{M \sqrt{n}} \leq \frac{\hat{p}_{min}}{2}$, Proposition \ref{prop:connection_clusterability_empiricalmargin} ensures that $X_1, \hdots, X_n$ is $\sqrt{p_{min} n}$-clusterable.

According to Theorem \ref{thm:classif_error_km}, on this probability event, at most $\frac{10}{p_{min}}$ points are misclassified by $\ckm$ compared to $\hat{\c}_n$. Thus, denoting by $n_j = n P_n V_j(\hat{\c}_n)$ and $\hat{n}_j = n(P_n V_j(\ckm))$, we may write 
\begin{align*}
\sum_{j=1}^{k} n_j \| \hat{c}_j - \hat{c}_{KM,j}\| & \leq \sum_{j=1}^{k} \| n_j \hat{c}_j - \hat{n}_j \hat{c}_{KM,j} \| + \left |n_j - \hat{n}_j \right |\|\hat{c}_{KM,j}\| \\
        & \leq \sum_{j=1}^{k} \left \| \sum_{i=1}^{n} X_i (\mathbbm{1}_{V_j(\hat{\c}_n)}(X_i) - \mathbbm{1}_{V_j(\ckm)}(X_i)) \right \|  +  \frac{20M}{p_{min}},    
\end{align*} 
since $\ckm$ and $\hat{\c}_n$ satisfy the centroid condition (Proposition \ref{prop:centroid_condition}). 
Thus,
\[
\sum_{j=1}^{k} n_j \| \hat{c}_j - \hat{c}_{KM,j}\| \leq \frac{30M}{p_{min}}.
\]
At last, since for all $j = 1, \hdots, k$, $\frac{n_j}{n} \geq \hat{p}_{min} \geq \frac{p_{min}}{2}$, we deduce that
\begin{align*}
\| \hat{\c}_n - \ckm \| \leq \frac{60M}{n p_{min}^2}.
\end{align*} 
\end{proof}
\subsection{Proof of Corollary \ref{cor:classification_error_GM}}
\begin{proof}
We recall that under the assumptions of Proposition \ref{prop:GM_margincondition}, $\c^*$ is unique and $P$ satisfies a margin condition with radius $\tilde{B}/8$. As in the proof of Theorem \ref{thm:margin_condition_gives_clusterability}, we assume that 
\[
\| \hat{\c}_n - \c^*\| \leq C \frac{\sqrt{1\log(n)}}{p_{min}\sqrt{n}}.
\]
This occurs with probability larger than $1-n^{-1}-e^{- \frac{n}{32M^4}\left ( (\delta - \frac{12kM^2}{\sqrt{n}}) \right )^2}$. It can be deduced from \cite[Corollary 2.1]{Biau08} that, with probability larger than $1-2e^{-x}$, 
\[
\sup_{\c \in \mathcal{B}(0,M)^k} \left | \hat{R}_{dist}(\c) - R_{dist}(\c) \right | \leq \frac{6kM^2 + 8M^2 \sqrt{2x}}{\sqrt{n}}.
\]
Therefore, for $n$ large enough, it holds
\[
\hat{R}_{dist}(\hat{\c}_n) \geq \frac{R_{dist}(\c^*)}{2}, 
\]
with probability larger than $1-1/n$. On this probability event, a large enough $n$ entails that every initialization of the Lloyd's algorithm is a good initialization. According to Theorem \ref{thm:fast_rates_distortion} and \ref{thm:margin_condition_gives_clusterability}, we may write
\[
\|\ckm - \c^*\| \leq C \frac{M \sqrt{\log(n)}}{\sqrt{n}}.
\]
Since, according to \cite[Lemma 4.2]{Levrard15}, $V_i(\c^*) \Delta V_i(\ckm) \subset \mathcal{B} \left ( N(\c^*, \frac{4\sqrt{2}M}{B}\|\ckm-\c^*\| \right )$, the margin condition entails that
\begin{align*}
R_{classif} \left ( \mathcal{C}(\ckm),\mathcal{C}(\c^*) \right ) \leq C M \sqrt{\frac{\log(n)}{n}} \leq C \sigma \sqrt{\frac{\log(n)}{n}}.
\end{align*}
Using Markov's inequality yields the same result in expectation. It remains to note that in the case $k=2$, $\Sigma_i = \sigma^2 I_d$ and $p_1 = p_2 = \frac{1}{2}$, though $\c^*$ may differ from $\m$, we have $\mathcal{C}(\c^*) = \mathcal{C}(\m)$.
\end{proof}


\end{document}